\documentclass[11pt]{article}
\usepackage[T1]{fontenc}
\usepackage{lmodern}
\usepackage[a4paper,margin=1in]{geometry}
\usepackage{amsmath,amssymb,amsthm,mathtools,microtype}
\usepackage{tikz}
\usetikzlibrary{arrows.meta,positioning,fit,backgrounds}
\usepackage[font=small,labelfont=bf]{caption}
\usepackage[hidelinks]{hyperref}
\usepackage{authblk}

\usepackage{microtype}
\newtheorem{theorem}{Theorem}[section]
\newtheorem{conjecture}[theorem]{Conjecture}
\newtheorem{lemma}[theorem]{Lemma}
\newtheorem{proposition}[theorem]{Proposition}
\newtheorem{corollary}[theorem]{Corollary}
\theoremstyle{remark}
\newtheorem{remark}[theorem]{Remark}
\numberwithin{equation}{section}
\newcommand{\one}{\mathbf{1}}
\title{The saturated spectral radius for complete graphs}
\author[1]{Zhengbo Chen}
\author[1]{Xiao-Dong Zhang\thanks{Corresponding author.}}
\affil[1]{School of Mathematical Sciences, MOE-LSC, SHL-MAC,
Shanghai Jiao Tong University, Shanghai 200240, China\\
\texttt{czb911@sjtu.edu.cn; xiaodong@sjtu.edu.cn}}
\date{}
\begin{document}
\maketitle
\begin{abstract}
A graph is $K_{r+1}$-saturated if it is $K_{r+1}$-free and adding
any missing edge creates a copy of $K_{r+1}$. Kim, Kim, Kostochka,
and O conjectured that $K_{r-1}\vee(n-r+1)K_1$ minimizes the
spectral radius among all $n$-vertex $K_{r+1}$-saturated graphs.
They proved the case $r=2$, and the cases $r=3$ and $r\in\{4,5\}$
were subsequently established by Kim, Kostochka, O, Shi, and Wang,
and by Wang and Hou, respectively. We settle the conjecture for all
$r\ge3$: if $n\ge r+1$ and $G$ is an $n$-vertex
$K_{r+1}$-saturated graph, then
\[
\rho(G)\ge
\frac{r-2+\sqrt{(r-2)^2+4(r-1)(n-r+1)}}{2},
\]
with equality if and only if $G\cong K_{r-1}\vee(n-r+1)K_1$.
We also prove O's local two-walk conjecture for every $r\ge2$:
\[
\sum_{w\in N_G(v)}d_G(w)
\ge (r-2)d_G(v)+(r-1)(n-r+1)
\qquad(v\in V(G)).
\]
If $G$ has no universal vertex, the inequality holds with the additional
term $(r-1)(r-2)$ on the right-hand side. This constant is best possible
uniformly in $n$ for every fixed $r$, and gives a strict improvement
when $r\ge3$. A corresponding spectral bound follows.
\end{abstract}

\section{Introduction}

For a fixed graph $F$, a graph $G$ is called \emph{$F$-saturated} if $G$ is $F$-free and $G+e$ contains a copy of $F$ for every $e\in E(\overline{G})$. The \emph{saturation number} $\operatorname{sat}(n,F)$ is the minimum number of edges in an $n$-vertex $F$-saturated graph.

A spectral analogue is obtained by minimizing the adjacency spectral radius. Writing $\rho(G)$ for the adjacency spectral radius of $G$, we define the \emph{spectral saturation number} by
\[
\operatorname{sat}_{\rho}(n,F)
=\min\bigl\{\rho(G): |V(G)|=n,\ G\text{ is }F\text{-saturated}\bigr\}.
\]
The associated extremal problems are to determine these two parameters and characterize the graphs attaining the respective minima. In this paper, we focus on the spectral saturation problem for complete graphs.

For integers $r\ge2$ and $n\ge r+1$, let
$S_{n,r}=K_{r-1}\vee(n-r+1)K_1$.
The theorem of Erd\H{o}s, Hajnal, and Moon~\cite{EHM1964} states
that every $n$-vertex $K_{r+1}$-saturated graph $G$ satisfies
\[
e(G)\ge (r-1)n-\binom{r}{2},
\]
with equality if and only if $G\cong S_{n,r}$.
Thus the minimum edge count and its unique extremal structure
are completely determined.

A natural spectral analogue asks whether $S_{n,r}$ also minimizes
the adjacency spectral radius among graphs in the same class.
This replaces the number of edges by a quantity that also reflects
how those edges are distributed among the vertices.
Its spectral radius is
\[
\rho(S_{n,r})=
\frac{r-2+\sqrt{(r-2)^2+4(r-1)(n-r+1)}}{2}.
\]
The classical edge bound alone does not yield this spectral bound:
combining it with $\rho(G)\ge 2e(G)/n$ gives only
\[
\rho(G)\ge 2(r-1)-\frac{r(r-1)}{n},
\]
which remains bounded for fixed $r$, whereas
$\rho(S_{n,r})\sim\sqrt{(r-1)n}$ as $n\to\infty$.
Obtaining the sharp spectral bound therefore requires additional
information about the local structure imposed by saturation.
Kim, Kim, Kostochka, and O~\cite{KKKO2020} proposed the following
conjecture.

\begin{conjecture}\label{conj:spectral}
Let $r\ge3$ and $n\ge r+1$. If $G$ is an $n$-vertex
$K_{r+1}$-saturated graph, then $\rho(G)\ge\rho(S_{n,r})$.
\end{conjecture}

Kim, Kim, Kostochka, and O~\cite{KKKO2020} first obtained an
asymptotically sharp lower bound for general $r$. More precisely,
by establishing a sharp lower bound on
$\sum_{v\in V(G)}d_G(v)^2$, they proved that every $n$-vertex
$K_{r+1}$-saturated graph $G$, with $n\ge r+1$ and $r\ge2$,
satisfies
\[
\rho(G)\ge
\sqrt{\frac{(r-1)(n-1)^2+(r-1)^2(n-r+1)}{n}}.
\]
For every fixed $r$, this bound is asymptotically sharp as
$n\to\infty$. They also settled the case $r=2$: every $n$-vertex
$K_3$-saturated graph satisfies $\rho(G)\ge\sqrt{n-1}$, with
equality if and only if $G$ is the star $K_{1,n-1}$ or a Moore
graph of diameter two.

Further progress established the conjectured exact bound
for several small values of $r$.
Kim, Kostochka, O, Shi, and Wang~\cite{KKOSW2023}
proved Conjecture~\ref{conj:spectral} for $r=3$ and showed
that $K_2\vee(n-2)K_1$ is the unique extremal graph.
Wang and Hou~\cite{WH2024} subsequently confirmed the
conjecture for $r=4$ and $r=5$.
Ai, Liu, O, and Zhang~\cite{ALOZ2025} independently established
the case $r=4$ by proving that every $n$-vertex $K_5$-saturated
graph $G$ satisfies
\[
\sum_{w\in N_G(v)}d_G(w)
\ge 2d_G(v)+3(n-3)
\qquad\text{for every }v\in V(G).
\]
Their result provides another proof of the sharp spectral
bound for $K_5$-saturated graphs, including the uniqueness
of the extremal graph $K_3\vee(n-3)K_1$.

Local walk counts provide a way to capture the additional structural
information needed for the spectral problem. If $A=A(G)$ is the
adjacency matrix and $\one$ is the all-ones vector, then
$(A\one)_v=d_G(v)$ and
\[
(A^2\one)_v=\sum_{w\in N_G(v)}d_G(w),
\]
the number of walks of length two starting at $v$.
A pointwise comparison between these quantities can be converted
into a quadratic inequality for the spectral radius by weighting
with a positive Perron eigenvector. Such a comparison also retains
information at individual vertices, which is useful for determining
the equality case. This motivates the following local two-walk
conjecture, recorded in~\cite[Problem 13B-2]{O2025}.

\begin{conjecture}\label{conj:local}
Let $r\ge2$ and $n\ge r+1$. If $G$ is an $n$-vertex
$K_{r+1}$-saturated graph, then every $v\in V(G)$ satisfies
\[
\sum_{w\in N_G(v)}d_G(w)
\ge (r-2)d_G(v)+(r-1)(n-r+1).
\]
\end{conjecture}

Kim, Kostochka, O, Shi, and Wang~\cite[Theorem 2.2]{KKOSW2023}
proved that Conjecture~\ref{conj:local} implies the spectral bound
in Conjecture~\ref{conj:spectral}. We establish the local conjecture
for every $r\ge2$ and determine the spectral equality case for
every $r\ge3$. Consequently, for $r\ge3$, the graph $S_{n,r}$
uniquely minimizes both the number of edges and the adjacency
spectral radius among $n$-vertex $K_{r+1}$-saturated graphs.
This gives an exact spectral counterpart to the classical
Erd\H{o}s--Hajnal--Moon theorem. For $r\ge3$, we also obtain a sharp improvement of the local
inequality for graphs with no universal vertex, together with
a corresponding stronger spectral bound.

\begin{theorem}\label{thm:spectral}
Let $r\ge3$ and $n\ge r+1$. If $G$ is an $n$-vertex
$K_{r+1}$-saturated graph, then
\[
\rho(G)\ge
\frac{r-2+\sqrt{(r-2)^2+4(r-1)(n-r+1)}}{2},
\]
with equality if and only if $G\cong K_{r-1}\vee(n-r+1)K_1$.
\end{theorem}

\begin{theorem}\label{thm:local}
Let $r\ge2$ and $n\ge r+1$. If $G$ is an $n$-vertex
$K_{r+1}$-saturated graph, then every $v\in V(G)$ satisfies
\begin{equation}\label{eq:local}
\sum_{w\in N_G(v)}d_G(w)
\ge(r-2)d_G(v)+(r-1)(n-r+1).
\end{equation}
If $G$ has no universal vertex, then
\begin{equation}\label{eq:strong}
\sum_{w\in N_G(v)}d_G(w)
\ge(r-2)d_G(v)+(r-1)(n-1)
\end{equation}
holds. For $r\ge3$, this improves the right-hand side of
\eqref{eq:local} by $(r-1)(r-2)$; for $r=2$, the two bounds coincide.
The additional constant $(r-1)(r-2)$ is best possible uniformly in $n$
for every fixed $r$.
\end{theorem}

The bound \eqref{eq:strong} also yields the following spectral consequence.

\begin{corollary}\label{cor:no-universal-spectral}
Let $r\ge2$ and $n\ge r+1$. If $G$ is an $n$-vertex
$K_{r+1}$-saturated graph with no universal vertex, then
\begin{equation}\label{eq:strong-spectral}
\rho(G)\ge
\frac{r-2+\sqrt{(r-2)^2+4(r-1)(n-1)}}{2}.
\end{equation}
For every $r\ge2$, equality is attained by $G=\overline{C_{2r+1}}$.
\end{corollary}

Our main new ingredient is a uniform
bound on independent set multiplicities. For comparison, the proof for
$K_5$-saturated graphs in~\cite[Section 4, Case 2]{ALOZ2025} estimates
the local two-walk count through a case analysis involving triangles
in a vertex neighborhood. In Proposition~\ref{prop:no-isolates}, we
work in the complement and control the overlaps of independent sets
using two inequalities. Hajnal's intersection--union lemma bounds
the sum of the sizes of the union and intersection of a family of
maximum independent sets. Sur\'anyi's inequality bounds the
neighborhood of that intersection inside an anti-neighborhood.
Together, they yield the multiplicity estimate
\eqref{eq:multiplicity} for every $r\ge2$.

Both ingredients apply without a connectedness hypothesis on the
complement. The final reduction therefore consists of deleting the
universal vertices, applying Proposition~\ref{prop:no-isolates} to the
remaining graph, and tracking the local surplus when the deleted
vertices are restored.

\section{Preliminaries}

All graphs considered in this paper are finite and simple. For a graph
$G$, let $V(G)$ and $E(G)$ denote its vertex set and edge set,
respectively, and write $e(G)=|E(G)|$. For a vertex $v\in V(G)$,
let $N_G(v)$ denote its open
neighborhood, let $N_G[v]=N_G(v)\cup\{v\}$ denote its closed
neighborhood, and write $d_G(v)=|N_G(v)|$ for its degree. A vertex
of an $n$-vertex graph is called
\emph{universal} if it has degree $n-1$. Throughout the paper, we write
$H=\overline G$ for the complement of $G$.

For two vertex-disjoint graphs $G_1$ and $G_2$, their \emph{join},
denoted by $G_1\vee G_2$, is obtained from their disjoint union by
adding all edges between $V(G_1)$ and $V(G_2)$. As usual, $K_r$
denotes the complete graph on $r$ vertices, $C_\ell$ denotes the cycle
on $\ell$ vertices, and $mK_1$ denotes the edgeless graph on $m$ vertices.

Given a graph $F$, a graph $G$ is called \emph{$F$-saturated} if
$G$ is $F$-free but $G+e$ contains a copy of $F$ for every
$e\in E(H)$. Let $A(G)$ denote the adjacency matrix of
$G$, and let $\rho(G)$ denote its spectral radius. For
$v\in V(G)$, the number of walks of length two starting at $v$ is
$\sum_{w\in N_G(v)}d_G(w)
$.

For $A\subseteq V(H)$, write
$N_H(A)=\bigcup_{a\in A}N_H(a)$ for its open neighborhood. An
independent set has no internal edges, and $\alpha(H)$ denotes the
maximum size of an independent set in $H$. We call $H$
$\alpha$-critical if $\alpha(H-e)=\alpha(H)+1$ for every
$e\in E(H)$; isolated vertices are allowed in this definition.

For $n\ge r+1$, a graph $G$ is $K_{r+1}$-saturated if and only if
its complement $H=\overline G$ is $\alpha$-critical and
$\alpha(H)=r$.
Indeed, adding an missing edge of $G$ is deleting an edge of $H$, and a
new $(r+1)$-clique is a new independent set of size $r+1$ in $H$.



\begin{lemma}[{\cite[Theorem~3.5]{Lovasz1994}}]
\label{lem:suranyi}
Let $H$ be an $\alpha$-critical graph without isolated vertices.
Then, for every independent set $A\subseteq V(H)$ and every $a\in A$,
\begin{equation}\label{eq:suranyi}
d_H(a)\le |N_H(A)|-|A|+1.
\end{equation}
\end{lemma}

\begin{lemma}[{\cite{Hajnal1965}}]
\label{lem:degree}
Let $G$ be a graph and let $H=\overline G$. If $H$ is
$\alpha$-critical without isolated vertices, then every $a\in V(H)$
satisfies
\begin{equation}\label{eq:degree}
d_H(a)\le |V(H)|-2\alpha(H)+1,
\qquad
d_G(a)\ge 2\alpha(H)-2.
\end{equation}
\end{lemma}

The first inequality is Hajnal's classical degree bound
\cite{Hajnal1965}; see also
\cite[Theorem~3.4]{Lovasz1994}.


\begin{lemma}\label{lem:expansion}
Let $H$ be $\alpha$-critical without isolated vertices, let $z\in V(H)$,
and let $W=H[V(H)\setminus N_H[z]]$.
Every independent set $I$ of $W$ satisfies
\[
|N_{W}(I)|\ge |I|.
\]
\end{lemma}

\begin{proof}
The set $I\cup\{z\}$ is independent in $H$. Moreover,
\[
N_H(I\cup\{z\})=N_H(z)\,\dot\cup\,N_{W}(I).
\]
To check this identity, a neighbor of $I$ outside $N_H(z)$ is
different from $z$, since $I\subseteq V(W)$, and thus belongs to $V(W)$.
Apply Lemma~\ref{lem:suranyi} to $I\cup\{z\}$ at $z$ to obtain
\[
d_H(z)\le d_H(z)+|N_{W}(I)|-|I|.
\]
Canceling $d_H(z)$ proves the assertion.
\end{proof}

\begin{lemma}
\label{lem:union}
Let $\mathcal I=\{I_1,\ldots,I_s\}$ be a family of maximum
independent sets of a graph $H$, where $|I_i|=\alpha(H)$. Then
\[
\left|\bigcap_{i=1}^s I_i\right|
+\left|\bigcup_{i=1}^s I_i\right|
\ge2\alpha(H).
\]
\end{lemma}

\begin{proof}
We proceed by induction on $s$. The case $s=1$ is immediate. Suppose
$s\ge2$. There are no edges between
$\bigcap_{i=1}^{s-1}I_i$ and $\bigcup_{i=1}^{s-1}I_i$: if
$c\in\bigcap_{i=1}^{s-1}I_i$ and
$d\in\bigcup_{i=1}^{s-1}I_i$, then $d\in I_i$ for some
$i<s$, while $c\in I_i$. Hence
\[
\left(I_s\cap\bigcup_{i=1}^{s-1}I_i\right)
\cup\left(\bigcap_{i=1}^{s-1}I_i\right)
\]
is independent, and therefore
$
\alpha(H)
\ge
\left|
\left(I_s\cap\bigcup_{i=1}^{s-1}I_i\right)
\cup\left(\bigcap_{i=1}^{s-1}I_i\right)
\right|=
\left|I_s\cap\bigcup_{i=1}^{s-1}I_i\right|
+
\left|\left(\bigcap_{i=1}^{s-1}I_i\right)\setminus I_s\right|$
$=
\alpha(H)
-\left|I_s\setminus\bigcup_{i=1}^{s-1}I_i\right|
+\left|\left(\bigcap_{i=1}^{s-1}I_i\right)\setminus I_s\right|,
$Which gives
$
\left|I_s\setminus\bigcup_{i=1}^{s-1}I_i\right|
\ge
\left|\left(\bigcap_{i=1}^{s-1}I_i\right)\setminus I_s\right|.
$
Thus the union gains at least as many vertices as the intersection
loses when $I_s$ is added. 

By the induction hypothesis,
\[
\left|\bigcap_{i=1}^sI_i\right|
+\left|\bigcup_{i=1}^sI_i\right|
\ge
\left|\bigcap_{i=1}^{s-1}I_i\right|
+\left|\bigcup_{i=1}^{s-1}I_i\right|
\ge2\alpha(H).
\]
\end{proof}
\section{Proofs of the main results}

\begin{proposition}\label{prop:no-isolates}
Let $H$ be an $\alpha$-critical graph of order $n$ with no
isolated vertices, and suppose that $\alpha(H)=r\ge2$.
Let $G=\overline H$. Then every vertex $v\in V(G)$ satisfies
\[
\sum_{w\in N_G(v)}d_G(w)
\ge (r-2)d_G(v)+(r-1)(n-1).
\]
\end{proposition}

\begin{proof}
Fix $v\in V(G)$ and let
\[
X=V(H)\setminus N_H[v]=N_G(v),\qquad Y=N_H(v).
\]
For each $y\in Y$, choose and fix a set $I_y\subseteq X$
of size $r-1$ such that $\{v,y\}\cup I_y$ is independent
in $H-vy$.There may be more than one such set for a given $y$;
we choose one arbitrarily and keep this choice fixed throughout
the proof. Such a choice is possible by the
$\alpha$-criticality of $H$.
In particular, $I_y$ is independent in $H$ and
$I_y\cap N_H(y)=\varnothing$. Moreover,
\[
|I_y|=r-1,\qquad
|\{v\}\cup I_y|=r,\qquad
|\{v,y\}\cup I_y|=r+1.
\]

Let
\[
\mathcal I=\{I_y:y\in Y\}
\]
be the family of chosen sets. Distinct indices may correspond
to the same set; the indices are retained in the definition
of $F_z$ below.
For each $z\in X$, define
\[
F_z=\{y\in Y:I_y\in\mathcal I,\ z\in I_y\}.
\]
We claim that
\begin{equation}\label{eq:multiplicity}
|F_z|\le d_G(z)-2r+3\qquad(z\in X).
\end{equation}

If $F_z=\varnothing$, then $|F_z|=0$.
Since $H$ is $\alpha$-critical without isolated vertices and
$\alpha(H)=r$, Lemma~\ref{lem:degree} gives
$d_H(z)\le n-2r+1$. Since $H=\overline G$, we obtain
\[
d_G(z)=n-1-d_H(z)
\ge n-1-(n-2r+1)=2r-2.
\]
Consequently,
\[
|F_z|=0<1\le d_G(z)-2r+3,
\]
so \eqref{eq:multiplicity} holds in this case.

Suppose $F_z\ne\varnothing$. Let
$W=H[V(H)\setminus N_H[z]]$ and let
\[
D=\bigcup_{y\in F_z}(I_y\setminus\{z\}),\qquad
C=\bigcap_{y\in F_z}(I_y\setminus\{z\}).
\]
The sets $\{v\}\cup I_y$, for $y\in F_z$, are maximum independent sets
of $H$. Their union and intersection are $\{v,z\}\cup D$ and
$\{v,z\}\cup C$, respectively. Lemma~\ref{lem:union} gives
\begin{equation}\label{eq:DC}
|D|+|C|\ge2r-4.
\end{equation}

The sets $\{v\}$, $F_z$, and $D$ are pairwise disjoint subsets of
$V(W)$. Moreover, $C$ has no neighbors in any of them. Indeed,
$C\subseteq X$, so no vertex of $C$ is adjacent to $v$; for every
$y\in F_z$, we have $C\subseteq I_y$, so no vertex of $C$ is
adjacent to $y$; and if $c\in C$ and $d\in D$, then $d\in I_y$ for
some $y\in F_z$, while $c\in I_y$, so $cd\notin E(H)$. Thus
$N_{W}(C)$ is disjoint from all three sets. Since $C$ is an independent
set of $W$, Lemma~\ref{lem:expansion} and \eqref{eq:DC} give
\begin{equation}\label{eq:decisive}
\begin{aligned}
d_G(z)=|V(W)|
&\ge1+|F_z|+|D|+|N_{W}(C)|\\
&\ge1+|F_z|+|D|+|C|\\
&\ge |F_z|+2r-3.
\end{aligned}
\end{equation}
This proves \eqref{eq:multiplicity} for every $r\ge2$.

Finally, counting incidences between $z\in X$ and the chosen $I_y$
gives $\sum_{z\in X}|F_z|=\sum_{y\in Y}|I_y|=(r-1)|Y|$. Summing
\eqref{eq:multiplicity} over $X$ therefore yields
\[
\begin{aligned}
\sum_{z\in N_G(v)}d_G(z)
&\ge(2r-3)|X|+(r-1)|Y|\\
&=(r-2)d_G(v)+(r-1)(n-1),
\end{aligned}
\]
as required.
\end{proof}
\begin{figure}[htbp]
\centering
\begin{tikzpicture}[font=\small,>=Latex]
  \draw[rounded corners=7pt,gray!70] (-0.5,-0.6) rectangle (12.8,2.8);
  \node[anchor=north west] at (-0.32,2.72) {$V(W)=V(H)\setminus N_H[z]$};
  \draw[rounded corners,fill=gray!8] (0,0) rectangle (1.1,1.65);
  \node at (.55,.83) {$\{v\}$};
  \draw[rounded corners,fill=gray!8] (1.65,0) rectangle (3.35,1.65);
  \node at (2.5,.83) {$F_z$};
  \draw[rounded corners,fill=blue!5] (3.9,0) rectangle (7.0,1.65);
  \node at (4.3,1.29) {$D$};
  \draw[rounded corners,fill=blue!15] (5.55,.4) rectangle (6.65,1.2);
  \node at (6.1,.8) {$C$};
  \draw[rounded corners,fill=green!9] (9.1,0) rectangle (12.25,1.65);
  \node at (10.68,1.1) {$N_{W}(C)$};
  \node[font=\footnotesize] at (10.68,.53) {at least $|C|$ vertices};
  \draw[->,thick,green!45!black] (6.65,.8)--(9.1,.8);
  \node[font=\footnotesize,align=center] at (8.06,2.15)
    {all neighbors of $C$\\inside $W$ lie here};
\end{tikzpicture}
\caption{The four disjoint sets counted in \eqref{eq:decisive}.
The arrow indicates containment of possible neighbors, not complete
adjacency. These sets need not cover $W$.}
\label{fig:count}
\end{figure}
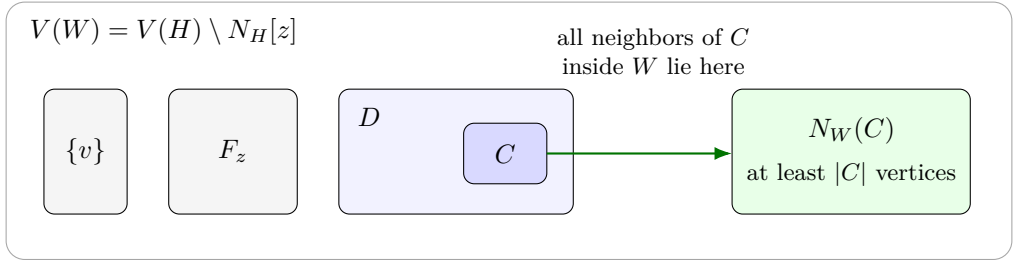

For an integer $r\ge1$, a graph $G$, and $v\in V(G)$, define the
local surplus by
\[
g_r(G,v)=\sum_{w\in N_G(v)}d_G(w)
-(r-2)d_G(v)-(r-1)(|V(G)|-r+1).
\]
In the reduction below, deleting $t$ universal vertices preserves
the surplus at every remaining vertex when the parameter is changed
from $r$ to $r-t$; see \eqref{eq:remove-universal}.If $G$ is $K_{r+1}$-saturated with no universal vertex and $r\ge2$,
Proposition~\ref{prop:no-isolates} gives
$g_r(G,v)\ge(r-1)(r-2)$ for every $v\in V(G)$.

\medskip
\begin{proof}[Proof of Theorem~\ref{thm:local}]
Let $t$ be the number of universal vertices of $G$. Delete them to
obtain a graph $G_0$ of order $m=n-t$, and let $k=r-t$.
Set $H_0=\overline{G_0}=H[V(G_0)]$.
Then $G=K_t\vee G_0$, and $H_0$ is obtained from $H$ by deleting
its $t$ isolated vertices. Hence $H_0$ is $\alpha$-critical without
isolated vertices, with $\alpha(H_0)=k$.
Since $n>r$, the graph $H$ has an edge, so $G_0$ is nonempty
and $k\ge1$.

For $v\in V(G_0)$, degrees and two-walk counts transform as
\[
\begin{aligned}
d_G(v)&=d_{G_0}(v)+t,\\
\sum_{w\in N_G(v)}d_G(w)
&=\sum_{w\in N_{G_0}(v)}d_{G_0}(w)+t\,d_{G_0}(v)+t(n-1).
\end{aligned}
\]
Substituting these expressions, together with $r=k+t$ and $n=m+t$,
gives the identity
\begin{equation}\label{eq:remove-universal}
\begin{aligned}
g_r(G,v)
&=\sum_{w\in N_{G_0}(v)}d_{G_0}(w)
  +t\,d_{G_0}(v)+t(n-1)
-(r-2)(d_{G_0}(v)+t)-(r-1)(n-r+1)\\
&=\sum_{w\in N_{G_0}(v)}d_{G_0}(w)
  -(k-2)d_{G_0}(v)+t(m-k+1)-(k+t-1)(m-k+1)\\
&=\sum_{w\in N_{G_0}(v)}d_{G_0}(w)
  -(k-2)d_{G_0}(v)-(k-1)(m-k+1)\\
&=g_k(G_0,v).
\end{aligned}
\end{equation}
If $k\ge2$, Proposition~\ref{prop:no-isolates} gives
\begin{equation}\label{eq:non-universal-surplus}
\begin{aligned}
g_r(G,v)=g_k(G_0,v)
&\ge(k-1)(m-1)-(k-1)(m-k+1)\\
&=(k-1)(k-2)\ge0.
\end{aligned}
\end{equation}
If $k=1$, then $H_0$ is complete, $G_0$ is edgeless, and
$g_1(G_0,v)=0$ directly.

For any universal vertex $u$, all other vertices are its neighbors,
so
\begin{equation}\label{eq:universal-surplus}
\begin{aligned}
g_r(G,u)
&=\sum_{w\in N_G(u)}d_G(w)
  -(r-2)d_G(u)-(r-1)(n-r+1)\\
&=2e(G)-2(r-1)n+r(r-1)\\
&=2e(G_0)+2tm+t(t-1)
-2(k+t-1)(m+t)+(k+t)(k+t-1)\\
&=2e(G_0)-2(k-1)m+k(k-1)\\
&\ge 2(k-1)m-2(k-1)m+k(k-1)\\
&=k(k-1)\ge0.
\end{aligned}
\end{equation}
The second equality uses $d_G(u)=n-1$ and
$\sum_{w\in N_G(u)}d_G(w)=2e(G)-(n-1)$, since $u$ is universal.
The third equality uses $n=m+t$, $r=k+t$, and
$e(G)=e(G_0)+tm+\binom t2$.
For the inequality, applying \eqref{eq:degree} to $H_0$ gives
$d_{G_0}(w)\ge2k-2$ for every $w\in V(G_0)$.
Summing over $V(G_0)$ yields $2e(G_0)\ge2(k-1)m$.
Finally, $k(k-1)\ge0$ because $k\ge1$.
This proves \eqref{eq:local} in all cases.

If $G$ has no universal vertex, Proposition~\ref{prop:no-isolates}
gives \eqref{eq:strong}. To prove sharpness, take
$G=\overline{C_{2r+1}}$, so that $H=C_{2r+1}$.
The graph $H$ is $\alpha$-critical
with $\alpha(H)=r$, and $G$ is $(2r-2)$-regular. Hence
\[
\sum_{w\in N_G(v)}d_G(w)=(2r-2)^2
=(r-2)(2r-2)+(r-1)\,2r,
\]
which is equality in \eqref{eq:strong}. Thus, for every fixed $r$,
the additional constant cannot be increased in a bound valid for all $n$.
\end{proof}

\begin{proof}[Proof of Theorem~\ref{thm:spectral}]
The graph $G$ is connected: every nonadjacent pair has a common
$(r-1)$-clique by saturation. Let $A$ be its adjacency matrix,
$\rho=\rho(G)$, and $q>0$ a Perron eigenvector.  Since
$(A\one)_v=d_G(v)$ and
$(A^2\one)_v=\sum_{w\in N_G(v)}d_G(w)$, the vector
\[
A^2\one-(r-2)A\one-(r-1)(n-r+1)\one
\]
has $v$th entry $g_r(G,v)$, which is nonnegative by
Theorem~\ref{thm:local}. Since $A$ is symmetric and $Aq=\rho q$,
\[
q^{\mathsf T}A=(Aq)^{\mathsf T}=\rho q^{\mathsf T},
\qquad
q^{\mathsf T}A^2=\rho^2q^{\mathsf T}.
\]
Consequently,
\begin{equation}\label{eq:perron}
\begin{aligned}
\sum_v q_v g_r(G,v)
&=q^{\mathsf T}\bigl(A^2\one-(r-2)A\one
  -(r-1)(n-r+1)\one\bigr)\\
&=\bigl(\rho^2-(r-2)\rho-(r-1)(n-r+1)\bigr)
  \sum_v q_v\\
&\ge0.
\end{aligned}
\end{equation}
The quadratic has a negative root and a positive root, and thus
$\rho$ is at least its positive root, as asserted.

Equality in the spectral bound forces $g_r(G,v)=0$ at every vertex,
because all $q_v$ are positive. Use $t$, $G_0$, and $k=r-t$ from the preceding
proof. If $t=0$, Proposition~\ref{prop:no-isolates} gives
$g_r(G,v)\ge(r-1)(r-2)>0$, since $r\ge3$.
If $t>0$ and $k\ge2$, \eqref{eq:universal-surplus} gives
$g_r(G,u)\ge k(k-1)>0$ at a universal vertex.
Both alternatives are impossible, so $k=1$. Consequently
$t=r-1$ and $G_0$ is edgeless, yielding
$G\cong K_{r-1}\vee(n-r+1)K_1$.

Conversely, let $G=K_{r-1}\vee(n-r+1)K_1$, with clique part
$C$ and independent-set part $I$. Every clique has order at most
$r$, while adding any missing edge $uv$ within $I$ creates a
$K_{r+1}$ on $C\cup\{u,v\}$. Thus $G$ is $K_{r+1}$-saturated.

The partition $(C,I)$ is equitable, with quotient matrix
\[
Q=\begin{pmatrix}
r-2 & n-r+1\\
r-1 & 0
\end{pmatrix}.
\]
Let $\lambda$ be the positive root of
\[
\lambda^2-(r-2)\lambda-(r-1)(n-r+1)=0.
\]
Then $(\lambda,r-1)^{\mathsf T}$ is a positive eigenvector of $Q$.
Assigning the value $\lambda$ to each vertex of $C$ and $r-1$
to each vertex of $I$ gives a positive vector $x$ satisfying
$A(G)x=\lambda x$. Since $G$ is connected, the
Perron--Frobenius theorem yields $\rho(G)=\lambda$, proving
equality in the claimed bound.
\end{proof}

\begin{proof}[Proof of Corollary~\ref{cor:no-universal-spectral}]
As in the proof of Theorem~\ref{thm:spectral}, $G$ is connected.
Let $A$ be its adjacency matrix, $\rho=\rho(G)$, and $q>0$ a Perron
eigenvector. By \eqref{eq:strong}, the vector
\[
A^2\one-(r-2)A\one-(r-1)(n-1)\one
\]
is entrywise nonnegative. Taking its inner product with $q$ and using
the symmetry of $A$, we obtain
\[
\bigl(\rho^2-(r-2)\rho-(r-1)(n-1)\bigr)\sum_v q_v\ge0.
\]
Taking the positive root gives \eqref{eq:strong-spectral}.
For $G=\overline{C_{2r+1}}$, we have $n=2r+1$ and $\rho(G)=2r-2$,
and the right-hand side of \eqref{eq:strong-spectral} equals $2r-2$.
\end{proof}

\begin{remark}
The local inequality and the spectral lower bound also hold for
$r=2$, but the spectral equality graph is not unique: $C_5$ already
gives an additional example, since $\rho(C_5)=2=\sqrt{5-1}$.
This is why Theorem~\ref{thm:spectral} assumes $r\ge3$.
\end{remark}

\section*{Acknowledgements}

This work was supported in part by the National Natural Science
Foundation of China under Grants 12371354 and W2521102. Additional
support was provided by the Montenegrin--Chinese Science and
Technology Cooperation Project under Grant 4-3 and by the Science
and Technology Commission of Shanghai Municipality under Grant
25LN3200600.


\begin{thebibliography}{9}

\bibitem{ALOZ2025}
J.~Ai, P.~Liu, S.~O, and J.~Zhang,
The minimum spectral radius of $tP_3$- or $K_5$-saturated
graphs via the number of $2$-walks,
\emph{Electron. J. Combin.} \textbf{32} (2025), Paper~P1.30.
\href{https://doi.org/10.37236/12492}
{doi:10.37236/12492}.

\bibitem{EHM1964}
P.~Erd\H{o}s, A.~Hajnal, and J.~W.~Moon,
A problem in graph theory,
\emph{Amer. Math. Monthly} \textbf{71} (1964), 1107--1110.

\bibitem{Hajnal1965}
A.~Hajnal,
A theorem on $k$-saturated graphs,
\emph{Canad. J. Math.} \textbf{17} (1965), 720--724.
\href{https://doi.org/10.4153/CJM-1965-072-1}
{doi:10.4153/CJM-1965-072-1}.

\bibitem{KKKO2020}
J.~Kim, S.-J.~Kim, A.~V.~Kostochka, and S.~O,
The minimum spectral radius of $K_{r+1}$-saturated graphs,
\emph{Discrete Math.} \textbf{343} (2020), Article~112068.
\href{https://doi.org/10.1016/j.disc.2020.112068}
{doi:10.1016/j.disc.2020.112068}.

\bibitem{KKOSW2023}
J.~Kim, A.~V.~Kostochka, S.~O, Y.~Shi, and Z.~Wang,
A sharp lower bound for the spectral radius in $K_4$-saturated graphs,
\emph{Discrete Math.} \textbf{346} (2023), Article~113231.
\href{https://doi.org/10.1016/j.disc.2022.113231}
{doi:10.1016/j.disc.2022.113231}.

\bibitem{LM2014}
V.~E.~Levit and E.~Mandrescu,
A Set and Collection Lemma,
\emph{Electron. J. Combin.} \textbf{21} (2014), Paper~P1.40.
\href{https://doi.org/10.37236/2514}
{doi:10.37236/2514}.

\bibitem{Lovasz1994}
L.~Lov\'asz,
Stable sets and polynomials,
\emph{Discrete Math.} \textbf{124} (1994), 137--153.

\bibitem{O2025}
S.~O,
On the minimum spectral radius in an $n$-vertex $K_{r+1}$-saturated graph,
Problem 13B-2, in \emph{Communications in Combinatorics},
Shanghai, December 12--15, 2025, p.~8.
\href{https://www.ibs.re.kr/ecopro/wp-content/uploads/2025/12/Communications-in-Combinatorics-pamphlet-20251212-15.pdf}{Official problem booklet}.

\bibitem{WH2024}
D.~Wang and Y.~Hou,
The minimum spectral radius for $K_{r+1}$-saturated graphs
with $r=4,5$,
\emph{Discrete Math.} \textbf{347} (2024), Article~114110.
\href{https://doi.org/10.1016/j.disc.2024.114110}
{doi:10.1016/j.disc.2024.114110}.

\end{thebibliography}
\end{document}